\documentclass[oneside,a4paper,11pt,notitlepage]{article}
\usepackage[left=0.8in, right=0.8in, top=1.5in, bottom=1.5in]{geometry}

\usepackage[T1]{fontenc} % codifica dei font
\usepackage[utf8]{inputenc} % lettere accentate da tastiera
\usepackage[english]{babel}
\usepackage{amssymb}
\usepackage{lipsum} % genera testo fittizio
\usepackage{lmodern}
\usepackage{amsmath}
\usepackage{mathrsfs}
\usepackage{cancel}
\usepackage{amsthm}
\usepackage{bm}
\usepackage{mathtools}
\usepackage{braket}
\usepackage{esint}
\newcommand{\abs}[1]{{\left|#1\right|}}
\newcommand{\norma}[1]{{\left\Vert#1\right\Vert}}
\newcommand{\dH}{\,d\mathcal{H}^{n-1}}

\newcommand{\nd}[1]{\frac{\partial#1}{\partial\nu}}

\usepackage{enumitem}
\usepackage{booktabs}
\usepackage{graphicx}
\usepackage{tikz}
\usetikzlibrary{patterns}
\usepackage{multicol}
\usepackage{caption}
\usepackage[skins,theorems]{tcolorbox}
\tcbset{highlight math style={enhanced,
colframe=black,colback=white,arc=0pt,boxrule=1pt}}
\def\XXint#1#2#3{{\setbox0=\hbox{$#1{#2#3}{\int}$}
    \vcenter{\hbox{$#2#3$}}\kern-.5\wd0}}

\theoremstyle{definition}
\newtheorem{definizione}{Definition}[section]
\theoremstyle{plain}
\newtheorem{theorem}{Theorem}[section]

\newtheorem{lemma}[theorem]{Lemma}

\theoremstyle{definition}
\newtheorem{esempio}{Example}[section]
\newtheorem{oss}[esempio]{Remark}

\newtheorem*{open*}{Open problems}
\DeclareMathOperator{\R}{\mathbb{R}}

\makeatletter
\newcommand{\myfootnote}[2]{\begingroup
	\def\@makefnmark{}%
	\addtocounter{footnote}{-1}%
	\footnote{\textbf{#1} #2}
	\endgroup}
\makeatother

   \usepackage{hyperref}
  \hypersetup{linktoc=none, bookmarksnumbered, colorlinks=true, linkcolor=magenta, citecolor=cyan}

\title{Local stability for a class of Saint-Venant type inequalities}
\author{Luca Barbato$^\ast$ and Francesco Salerno}
\date{}

\newcommand{\Addresses}{{% additional braces for segregating \footnotesize 
\bigskip 
  \footnotesize 
\noindent \textit{E-mail address}, Luca~Barbato (corresponding author): \texttt{l.barbato@ssmeridionale.it} 

     \medskip 

 \noindent \textit{E-mail address}, Francesco~Salerno: \texttt{f.salerno@ssmeridionale.it}

   \medskip 
 
  \textsc{Mathematical and Physical Sciences for Advanced Materials and Technologies, Scuola Superiore Meridionale, Largo San Marcellino 10, 80138 Napoli, Italy. }

 \par\nopagebreak 

}} 
\begin{document}
\maketitle

\begin{abstract}
We establish a local stability result for a class of Saint–Venant type inequalities. Given the solution $u$ of the Dirichlet torsion problem in a domain $\Omega$, we consider shape functionals $\mathcal{J}(\Omega)$ involving the integral of $j(u)$, where $j$ is convex and satisfies suitable structural assumptions. By Talenti's comparison principle, balls maximize $\mathcal{J}$ among sets of prescribed measure. We prove that this extremal property is stable in the class of nearly spherical sets: the deficit from the optimal value controls the square of the $H^{1/2}$-norm of the boundary perturbation. The argument relies on shape derivative techniques, including the computation of the second variation and the introduction of an adjoint state. As applications, the result covers several relevant examples, including the torsional rigidity, $L^p$-norms of the torsion function for $p\ge 2$, and Moser–Trudinger functional in dimension two.
\end{abstract}
\section{Introduction}
In their seminal paper \cite{BDPV}, Brasco, De Philippis and Velichkov proved a sharp quantitative version of the Faber--Krahn inequality for the first eigenvalue of the Dirichlet Laplacian. More precisely, for an open set $\Omega\subset\mathbb{R}^n$, $n\geq 2$, with finite measure, the first eigenvalue is defined by
\begin{equation*}
    \lambda_1(\Omega)=\min_{u\in H^1_0(\Omega)}\left\{\int_\Omega |\nabla u|^2\,dx :\|u\|_{L^2(\Omega)}=1\right\}.
\end{equation*}
It is well known, by the Faber--Krahn inequality, that $\lambda_1(\Omega)$ is minimized among sets of fixed measure by the ball $\Omega^\sharp$ having the same Lebesgue measure as $\Omega$. Brasco, De Philippis and Velichkov strengthened this result by proving a quantitative stability estimate of the form
\begin{equation*}
    \lambda_1(\Omega)-\lambda_1(\Omega^\sharp)\geq c(n)\, d(\Omega,\Omega^\sharp)^2,
\end{equation*}
where $d(\Omega,\Omega^\sharp)$ denotes a suitable measure of the deviation of $\Omega$ from spherical symmetry. Actually this result was the first with sharp exponent on the asymmetry, while results with different exponent were already available in literature (see for instance \cite{BW, FMP, N}).

A key step in their analysis consists in reducing the problem to a quantitative version of the Saint--Venant inequality. The central object in this setting is the torsional rigidity, defined as the $L^1$-norm of the solution to the Dirichlet problem
\begin{equation}\label{Problema(u)}
    \begin{cases}
        -\Delta u=1 & \text{in } \Omega,\\
        u=0 & \text{on } \partial\Omega.
    \end{cases}
\end{equation}
Namely,
\[
    T(\Omega)=\int_\Omega u\,dx.
\]
In the planar case, this quantity represents the resistance to torsion of an elastic beam whose cross-section is $\Omega$. The classical Saint--Venant inequality asserts that, among sets of fixed measure, $T(\Omega)$ is maximized by the ball $\Omega^\sharp$. Brasco, De Philippis and Velichkov proved the corresponding quantitative stability estimate
\begin{equation*}
    T(\Omega^\sharp)-T(\Omega)\geq c(n)\, d(\Omega,\Omega^\sharp)^2.
\end{equation*}
This result belongs to a broad literature on quantitative inequalities for spectral and geometric functionals \cite{AB2,BFNT,CCLMP,F}, see also \cite[Chapter 7]{H}. Related stability questions have also been investigated in the more general context of solutions to PDEs \cite{AL,AB,ABCMP,ABMP,MS}. 

The classical Saint--Venant inequality can be derived from Talenti's celebrated comparison principle \cite{T}, which compares the solution of \eqref{Problema(u)} with the solution of a suitable symmetrized problem posed on $\Omega^\sharp$. This comparison also yields the so-called mass comparison principle; see \cite{ALT}. As a consequence, the ball is extremal, among sets of prescribed measure, for a broad class of shape functionals of the form
\begin{equation*}
    \mathcal{J}(\Omega)=\int_\Omega j(u)\,dx,
\end{equation*}
where $j:\mathbb{R}\to\mathbb{R}$ is a convex, nonnegative, Lipschitz continuous function satisfying $j(0)=0$.

The aim of the present paper is to extend the quantitative stability result discussed above to this more general class of functionals, at least within the class of nearly spherical sets. More precisely, we consider bounded sets $\Omega\subset\mathbb{R}^n$ whose boundary is a small normal perturbation of the unit sphere. Namely, there exists a function $\varphi\in C^{2,\gamma}(\partial B_1)$, with $    \|\varphi\|_{L^\infty(\partial B_1)}\leq \frac12,
$ such that
\begin{equation*}
    \partial\Omega=\left\{ x\in\mathbb{R}^n :x=(1+\varphi(y))y,\quad y\in\partial B_1\right\}.
\end{equation*}

Our approach is based on the method of shape derivatives; we refer to \cite{HP} for a comprehensive introduction to the subject. In this framework, one considers a volume-preserving deformation $\{\Omega(t)\}_{t\geq 0}$ of the unit ball, with $\Omega(0)=B_1$, and studies the corresponding Gâteaux derivative of the functional. Formally, this derivative is given by
\begin{equation*}
    d\mathcal{J}(B_1)[V]=\lim_{t\to0^+}\frac{\mathcal{J}(\Omega(t))-\mathcal{J}(B_1)}{t},
\end{equation*}
where $V$ denotes the velocity field associated with the deformation.

In particular, we shall use a family of sets $\Omega(t)$, starting from $\Omega(0)=B_1$, whose evolution is described by the following lemma. This result is taken from \cite[Lemma A.1]{BDPV}.
\begin{lemma}\label{lemmaA.1}
    Given $\gamma\in(0,1]$ there exists $\delta_e=\delta_e(n,\gamma)>0$ and a modulus of continuity $\omega_e$ such that, for every nearly spherical set $\Omega$ parametrized by $\varphi$ with $\|\varphi\|_{C^{2,\gamma}(\partial B_1)}\leq\delta_e$ and $\abs{\Omega}=\abs{B_1}$, we can find an autonomous vector field $V_\varphi$ for which the following holds true:
    \begin{itemize}
        \item [(i)] $\mathrm{div}\,V_\varphi=0$ in a $\delta_e$-neighborhood of $\partial B_1$,
        \item [(ii)] if $\Phi_t:=\Phi(t,x)$ is the flow of $V_\varphi$, i.e.
        \begin{equation*}
            \partial_t \Phi_t=V_\varphi(\Phi_t),\quad \Phi_0(x)=x,
        \end{equation*}
        then $\Phi_1(\partial B_1)=\partial\Omega$ and $\abs{\Phi_t(B_1)}=\abs{B_1}$ for all $t\in[0,1]$,
        \item [(iii)] We have  
        \begin{equation*}
            \|\Phi_t-Id\|_{C^{2,\gamma}}\leq\omega_e(\|\varphi\|_{C^{2,\gamma}(\partial B_1)})\quad \text{for every }t\in[0,1],
        \end{equation*}
        \begin{equation*}
            \|\varphi-(V_\varphi\cdot\nu_0)\|_{H^{1/2}(\partial B_1)}\leq\omega_e(\|\varphi\|_{L^\infty(\partial B_1)})\|\varphi\|_{H^{1/2}(\partial B_1)},
        \end{equation*}
        and
        \begin{equation*}
            (V_\varphi\cdot\theta)\circ\Phi_t-V_\varphi\cdot\nu_0=(V_\varphi\cdot\nu_0)\psi_t \quad \text{on }\partial B_1,
        \end{equation*}
        with $\|\psi_t\|_{C^{2,\gamma}(\partial B_1)}\leq\omega_e(\|\varphi\|_{C^{2,\gamma}(\partial B_1)})$.
    \end{itemize}
\end{lemma}
For notational simplicity, throughout the paper we shall omit the subscript and write $V$ instead of $V_\varphi$. Therefore, by additionally assuming that
\begin{itemize}
    \item $j$ is of class $C^2$;
    \item $j'$ is positive in $(0,\delta)$ for some $\delta\geq\frac{1}{2n}$;
    \item $j''$ is bounded.
\end{itemize}
Our main result states as follows.
\begin{theorem}[Stability for nearly spherical sets]\label{MainTh}
    Let $0<\gamma\leq1$. Then there exist $\delta_c=\delta_c(n,\gamma,j)$ and $C(n,j)>0$ such that if $\Omega$ is a nearly spherical set of class $C^{2,\gamma}$ parametrized by $\varphi$ with 
    \begin{equation*}
        \|\varphi\|_{C^{2,\gamma}}\leq\delta_c,\quad \abs{\Omega}=\abs{B_1}, \quad \text{and}\quad x_\Omega=0,
    \end{equation*}
    then 
    \begin{equation*}
        \mathcal{J}(B_1)-\mathcal{J}(\Omega)\geq C(n,j)\|\varphi\|_{H^{1/2}(\partial B_1)}^2.
    \end{equation*}
\end{theorem}
We emphasize that, although our result is restricted to the class of nearly spherical sets, the stability deficit obtained in terms of the $H^{1/2}$-norm captures both volumetric and oscillatory information, since it controls not only the size of the perturbation but also the oscillations of the boundary.

The main difference with respect to \cite{BDPV}, which corresponds to the special case $j(s)=s$, lies in the structure of the second variation. Since the functional considered here is no longer simply the torsional rigidity, its second shape derivative contains additional terms depending on the second material derivative of the state function. These terms have to be either controlled or eliminated from the expression of the second variation. For this purpose, we introduce an adjoint state, namely the solution of a suitable auxiliary boundary value problem, which allows us to remove precisely this dependence and to obtain a tractable second-order expansion.

The paper is organized as follows. In Section~\ref{Notations} we introduce the notation and the main analytical tools. Section~\ref{derivatives} is devoted to the computation of first and second order shape derivatives. Finally, in Section~\ref{stability} we prove improved continuity and coercivity which are the main tools in the proof of Theorem \ref{MainTh}. Moreover, in the Appendix~\ref{listafunzione} we exhibit some functionals for which our stability result applies.

\section{Notations and Preliminaries}\label{Notations}
\subsection{Rearrangements and comparison principles}

Throughout this section, $\Omega\subset\mathbb{R}^n$ denotes a measurable set with finite measure. We write $\abs{E}$ for the Lebesgue measure of a measurable set $E$ and $\omega_n$ for the measure of the unit ball in $\mathbb{R}^n$. If $E\subset\mathbb{R}^n$ is measurable, we denote by $E^\sharp$ the open ball centered at the origin such that
\[
    \abs{E^\sharp}=\abs{E}.
\]

We recall the basic notation concerning rearrangements. 

\begin{definizione}[Decreasing rearrangement]
Let $u:\Omega\to\mathbb{R}$ be a measurable function. Its distribution function is defined by
\[
    \mu_u(t):=\abs{\{x\in\Omega:\abs{u(x)}>t\}},\qquad t\geq 0.
\]
The decreasing rearrangement of $u$ is the function $u^\ast:(0,\abs{\Omega})\to[0,+\infty]$ given by
\[
    u^\ast(s):=\inf\{t\geq 0:\mu_u(t)\leq s\}.
\]
Equivalently, $u^\ast$ is the unique nonincreasing, right-continuous function which is equimeasurable with $\abs{u}$.
\end{definizione}

For a comprehensive account of the main properties of rearrangements and their applications to variational problems, we refer to \cite{K}.

Rearrangement techniques are a classical tool in the proof of sharp geometric inequalities. In particular, they were used by P\'olya to prove Saint-Venant's inequality for the torsional rigidity \cite{P}, which states that, among all domains with prescribed measure, the ball maximizes the torsional rigidity. A different and very influential approach was later introduced by Talenti \cite{T}. More precisely, let $u$ be the solution of the torsion problem
\[
    \begin{cases}
        -\Delta u=1 & \text{in }\Omega,\\
        u=0 & \text{on }\partial\Omega,
    \end{cases}
\]
and let $v$ be the solution of the corresponding problem in the ball $\Omega^\sharp$. Talenti's comparison principle gives the sharp pointwise estimate
\[
    u^\ast(s)\leq v^\ast(s),\qquad s\in(0,\abs{\Omega}).
\]
In particular, this implies the integral comparison
\[
    \int_\Omega u\,dx=\int_0^{\abs{\Omega}}u^\ast(s)\,ds\leq\int_0^{\abs{\Omega}}v^\ast(s)\,ds=\int_{\Omega^\sharp}v\,dx,
\]
which is precisely Saint-Venant's inequality.

We shall use the following standard characterization of comparison by mass concentration; see, for instance, \cite{ALT}.

\begin{lemma}
Let $f,g\in L^1(\Omega)$ be nonnegative functions. Then the following conditions are equivalent:
\begin{enumerate}
    \item for every $t\in[0,\abs{\Omega}]$,
    \[
        \int_0^t f^\ast(s)\,ds\leq\int_0^t g^\ast(s)\,ds;
    \]
    
    \item for every nonnegative $\varphi\in L^\infty(\Omega)$,
    \[
        \int_\Omega f\varphi\,dx\leq\int_0^{\abs{\Omega}} g^\ast(s)\varphi^\ast(s)\,ds;
    \]
    
    \item for every nonnegative $\varphi\in L^\infty(\Omega)$,
    \[
        \int_0^{\abs{\Omega}} f^\ast(s)\varphi^\ast(s)\,ds\leq\int_0^{\abs{\Omega}} g^\ast(s)\varphi^\ast(s)\,ds;
    \]
    
    \item for every nonnegative Lipschitz convex function $F$ such that $F(0)=0$,
    \[
        \int_\Omega F(f)\,dx\leq\int_\Omega F(g)\,dx.
    \]
\end{enumerate}
\end{lemma}
In this terminology, Talenti's theorem asserts that the solution of the torsion problem in a general domain is less concentrated than the corresponding radially symmetric solution in the ball of the same measure. In the applications below, this will provide the maximality of the ball for the functional $\mathcal{J}(\Omega)$.
\subsection{Harmonic extension and spherical harmonics}
\begin{definizione}
Given a function $\varphi:\partial B_1\to \R$ we define
\[
\|\varphi\|_{H^{1/2}(\partial B_1)}^2:=\int_{\partial B_1} \varphi^2 \,\dH+\int_{B_1} |\nabla H(\varphi)|^2\,dx,
\]
where $H(\varphi)$ is the $W^{1,2}$ harmonic extension of $\varphi$, i.e.
\[
    \begin{cases}
        -\Delta H(\varphi)=0 & \text{ in } B_1\\
        H(\varphi)=\varphi & \text{ on } \partial B_1.
    \end{cases}
\]
\end{definizione}
In particular $H^{1/2}(\partial B_1)$, equipped with this norm, is a Hilbert space. We also recall that the trace Poincaré--Wirtinger inequality.
\[
\int_{\partial B_1}\abs{w-\int_{\partial B_1}w}^2\,\dH\leq\int_{B_1} \abs{\nabla w}^2\,dx,\qquad w\in W^{1,2}(B_1),
\]
implies
\[
\norma{\nabla H(\varphi)}_{L^2(B_1)}\leq\norma{\varphi}_{H^{1/2}(\partial B_1)}\leq\sqrt{2}\norma{\nabla H(\varphi)}_{L^2(B_1)}\quad\text{for every } \varphi \text{ s.t. } \int_{\partial B_1}\varphi\,\dH=0.
\]

We now recall some basic facts about spherical harmonics; these functions form a Hilbert basis of $L^2(\partial B_1)$, and hence also provides the standard spectral decomposition in $H^{1/2}(\partial B_1)$ (see, for instance, \cite[Remark 4.1]{BDR}).
In our setting, assuming that $\xi$ is orthogonal to spherical harmonics of degree $0$ and $1$, we may write
\begin{equation*}
    \xi(\theta)=\sum_{k\geq2}c_k Y_k(\theta)\quad\text{where }\theta\in\mathbb{S}^{n-1}.
\end{equation*}
Its harmonic extension is then given by
\begin{equation*}
    H(\xi)(x)=\sum_{k\geq2}c_k\abs{x}^kY_k\left(\frac{x}{\abs{x}}\right)\quad\text{where }x\in B_1.
\end{equation*}
By orthonormality, it follows that
\begin{equation*}
\int_{\partial B_1}\xi^2\,\dH=\sum_{k\geq2} c_{k}^2.
\end{equation*}
Moreover, using polar coordinates, we obtain
\begin{equation*}
\int_{B_1}H(\xi)^2\,dx
=\sum_{k\geq2} c_{k}^2\int_0^1 r^{2k+n-1}\,dr=\sum_{k\geq2}\frac{c_{k}^2}{2k+n}\leq\frac{1}{n}\int_{\partial B_1}\xi^2\,\dH.
\end{equation*}
Finally, since $H(\xi)$ is harmonic, its Dirichlet energy can be expressed in terms of the Steklov eigenvalues:
\[
\int_{B_1} \abs{\nabla H(\xi)}^2\,dx=\sum_{k\geq2}\sigma_k(B_1)c_{k}^2=\sum_{k\geq2}kc_{k}^2,
\]
where the last equality follows from the fact that the Steklov eigenvalue associated with spherical harmonics of degree $k$ on the unit ball is precisely $k$.
\subsection{Adjoint state}
In this subsection, we show how the introduction of an auxiliary problem, namely the \emph{adjoint problem}, can sometimes simplify the expression of the derivatives of shape functionals. We explain this idea for functions that vanish on the boundary of $\Omega$, although the technique is not specific to Dirichlet boundary conditions (see \cite[\S 5.8]{HP} for a more comprehensive discussion of the topic).\\
The main idea is to eliminate the derivative $\dot{u}_t$ from the expression of the derivative of the functional. More precisely, when dealing with a function $u_t$ that vanishes on $\partial\Omega_t$, we obtain the following information on $\dot{u}_t|_{\partial\Omega_t}$:
\begin{equation*}
    \dot{u}_t=-\nabla u_t\cdot V \quad \text{ on } \partial\Omega_t,
\end{equation*}
where $V$ is the velocity vector field associated with the deformation under consideration. It then remains to handle the occurrences of $\dot{u}_t$ inside the integrals over $\Omega_t$. Having explained the idea and motivation behind the adjoint problem, let us now define it.\\
Consider the following shape functional:
\begin{equation*}
    g(t)=\mathcal{G}(\Omega_t)=\int_{\Omega_t} G(x,u_t(x),\nabla u_t(x))\,dx,
\end{equation*}
where $G=G(x,u,q)$ is a sufficiently regular function. Then, using Hadamard's formula, we obtain
\begin{equation}\label{SECADJ}
    g'(t)=\int_{\Omega_t}[\dot{u}_t\partial_u G+\nabla_q G\cdot\nabla\dot{u}_t]\,dx+\int_{\partial\Omega_t}G \, V\cdot\nu_t\,\dH.
\end{equation}
From a heuristic point of view, we now want to replace $\dot{u}_t$ with a test function $v$ in the expression of $g'(t)$ in order to derive a variational formulation for the equation satisfied by $p_t$. This leads to the following adjoint problem:
\begin{equation*}
    \begin{cases}
        -\Delta p_t+p_t=\partial_u G-\mathrm{div}_x(\nabla_q G) & \text{ in }\Omega_t,\\
        p_t=0 & \text{ on } \partial\Omega_t.
    \end{cases}
\end{equation*}
Multiplying the equation by $\dot{u}_t$ and integrating by parts, we obtain
\begin{equation*}
    -\int_{\partial\Omega_t}\dot{u}_t\frac{\partial p_t}{\partial\nu_t}\,\dH
    =
    \int_{\Omega_t}[\dot{u}_t\partial_uG+\nabla \dot{u}_t\cdot\nabla_q G]\,dx
    -\int_{\partial\Omega_t}\dot{u}_t\nabla_q G\cdot\nu_t\,\dH.
\end{equation*}
Using \eqref{SECADJ}, we then deduce
\begin{equation}\label{SECADJ2}
    g'(t)=\int_{\partial\Omega_t}(V\cdot\nu_t)\left[\frac{\partial u_t}{\partial\nu_t}\frac{\partial p_t}{\partial\nu_t}+G-\frac{\partial u_t}{\partial\nu_t}\nabla_q G\cdot\nu_t \right]\,\dH.
\end{equation}
Comparing \eqref{SECADJ} and \eqref{SECADJ2}, we see that all contributions involving $\dot{u}_t$ in $\Omega_t$ have been eliminated. Furthermore, this will allow us to exclude a priori any dependence on $\ddot{u}_t$ in the second derivative of the functional for terms defined on $\Omega$.
\section{First and Second variation}\label{derivatives}
Defining $f(t):=\mathcal{J}(\Omega(t))$, in the next results we compute the first and second variation of the shape functional.

Moreover, we emphasize that, here and in what follows, since all the functions under consideration are radial in $t=0$, whenever $\abs{\nabla\cdot}$ appears outside an integral, it is to be understood as the constant value it attains on $\partial B_1$.
\begin{lemma}[First Variation]
    Under the assumptions of Lemma \ref{lemmaA.1}, we have
    \begin{equation*}
        f'(t)=\int_{\partial\Omega(t)}\abs{\nabla u_t}\abs{\nabla p_t}(V\cdot\nu_t)\,d\mathcal{H}^{n-1}.
    \end{equation*}
\end{lemma}
\begin{proof}
    Recalling the well-known \emph{Hadamard's formula}
\begin{equation*}
    \frac{d}{dt}\int_{\Omega(t)}f(x,t)\,dx=\int_{\Omega(t)} [f_t(x,t)+\mathrm{div}(fV)]\,dx,
\end{equation*}
we can compute the first derivative as follows
\begin{equation*}
    f'(t)=\int_{\Omega(t)}j'(u_t)\dot{u}_t\,dx+\int_{\partial\Omega(t)}j(u_t)(V\cdot \nu_t)\,d\mathcal{H}^{n-1}.
\end{equation*}
From the boundary data in \eqref{Problema(u)}, since the vector field $V$ is divergence-free in a neighborhood of $\partial B_1$ (where $\partial\Omega(t)$ is contained), we have that the boundary integral is zero. Looking at \eqref{Problema(u)} we remark that $\dot{u}$ solves the following problem
\begin{equation}\label{Problema(u.)}
    \begin{cases}
        -\Delta \dot{u}_t=0 & \text{ in }\Omega(t)\\
        \dot{u}_t=-\nabla u_t\cdot V & \text{ on } \partial\Omega(t).
    \end{cases}
\end{equation}
We want now to introduce an auxiliary problem, known in literature as \emph{adjoint problem} (see \S\ref{Notations} for more details), which will allow us to remove the dependence on $\dot{u}$ in $f'(t)$, and therefore remove the dependence on $\ddot{u}$ in $f''(t)$. More precisely we consider the following problem
\begin{equation}\label{Problema(p)}
    \begin{cases}
        -\Delta p_t= j'(u_t) & \text{ in }\Omega(t)\\
        p_t=0 & \text{ on } \partial\Omega(t).
    \end{cases}
\end{equation}
With this notation, we can rewrite $f'(t)$ as follows
\begin{equation*}
    f'(t)=\int_{\Omega(t)}-\Delta p_t \dot{u}_t\,dx=-\int_{\partial \Omega(t)}\dot{u}_t\nabla p_t\cdot\nu_t\,d\mathcal{H}^{n-1}=\int_{\partial\Omega(t)}\abs{\nabla u_t}\abs{\nabla p_t}(V\cdot\nu_t)\,d\mathcal{H}^{n-1},
\end{equation*}
where we have applied \emph{Green's formula} and \eqref{Problema(u.)}-\eqref{Problema(p)}.
\end{proof}
 \begin{oss}
     Recalling that $\Omega(0)=B_1$ and $u_0$, $p_0$ are radial functions, we have that
    \begin{equation*}
        f'(0)=\abs{\nabla u_0}\abs{\nabla p_0}\int_{\partial B_1}V\cdot \nu_0\,d\mathcal{H}^{n-1}=0,
    \end{equation*}
    which is not surprising since, from Talenti's inequality, we already know that the ball is a maximizer for $\mathcal{J}(\cdot)$.
 \end{oss}

\begin{lemma}[Second Variation]
    Under the assumptions of Lemma \ref{lemmaA.1}, we have
    \begin{equation}\label{dersec(3)}
    \begin{split}
        f''(t)=&-2\int_{\Omega(t)}\nabla \dot{u}_t\cdot\nabla\dot{p}_t \,dx+\int_{\Omega(t)}j''(u_t)\dot{u}_t^2\,dx+\int_{\partial\Omega(t)}\frac{\abs{\nabla p_t}}{\abs{\nabla u_t}}\langle D^2u_t.\nabla u_t,V^\tau\rangle(V\cdot\nu_t)\,d\mathcal{H}^{n-1}\\
        &+\int_{\partial\Omega(t)}\frac{\abs{\nabla u_t}}{\abs{\nabla p_t}}\langle D^2p_t.\nabla p_t,V^\tau\rangle(V\cdot\nu_t)\,d\mathcal{H}^{n-1}+\int_{\partial\Omega(t)}\abs{\nabla p_t}(V\cdot\nu_t)^2\,d\mathcal{H}^{n-1}\\
        &+\int_{\partial\Omega(t)}\abs{\nabla u_t}j'(u_t)(V\cdot\nu_t)^2\,d\mathcal{H}^{n-1}-2\int_{\partial\Omega(t)}\abs{\nabla u_t}\abs{\nabla p_t}\mathscr{H}_{\partial\Omega(t)}(V\cdot\nu_t)^2\,d\mathcal{H}^{n-1}.
    \end{split}
\end{equation}
\end{lemma}
\begin{proof}
  Since we want to apply the Hadamard's formula to $f'(t)$, we rewrite the boundary integral, using the divergence Theorem, as follows
\begin{equation*}
    f'(t)=\int_{\Omega(t)}\mathrm{div}(\abs{\nabla u_t}\abs{\nabla p_t}V)\,dx.
\end{equation*}
Then we can compute the second derivative
\begin{equation}\label{dersec(1)}
    f''(t)=\int_{\Omega(t)}\frac{\partial}{\partial t}\mathrm{div}(\abs{\nabla u_t}\abs{\nabla p_t}V)\,dx+\int_{\partial\Omega(t)}\mathrm{div}(\abs{\nabla u_t}\abs{\nabla p_t}V)(V\cdot\nu_t)\,d\mathcal{H}^{n-1}.
\end{equation}
Regarding the first integral, since $V$ is an autonomous vector field, applying again the divergence Theorem, we have that
\begin{equation}\label{parte1}
    \begin{split}
        \int_{\Omega(t)}\frac{\partial}{\partial t}\mathrm{div}(\abs{\nabla u_t}\abs{\nabla p_t}V)\,dx&=\int_{\Omega(t)}\mathrm{div}\left(\frac{\partial}{\partial t}(\abs{\nabla u_t}\abs{\nabla p_t})V\right)\,dx=\int_{\partial\Omega(t)}\frac{\partial}{\partial t}(\abs{\nabla u_t}\abs{\nabla p_t})(V\cdot\nu_t)\,d\mathcal{H}^{n-1}\\
        &=\int_{\partial\Omega(t)}\left[\abs{\nabla p_t}\frac{\nabla u_t}{\abs{\nabla u_t}}\cdot\nabla \dot{u}_t+\abs{\nabla u_t}\frac{\nabla p_t}{\abs{\nabla p_t}}\cdot\nabla \dot{p}_t \right](V\cdot\nu_t)\,d\mathcal{H}^{n-1}\\
        &=-\int_{\partial\Omega(t)}\left[\abs{\nabla p_t}\frac{\partial\dot{u}_t}{\partial\nu_t}+\abs{\nabla u_t}\frac{\partial\dot{p}_t}{\partial\nu_t} \right](V\cdot\nu_t)\,d\mathcal{H}^{n-1}.
    \end{split}
\end{equation}
Now we want to rewrite the second term in \eqref{dersec(1)}. For simplicity we recall that
\begin{equation*}
    \nabla(\abs{\nabla g})\cdot V=\frac{\langle D^2g.\nabla g, V\rangle}{\abs{\nabla g}}.
\end{equation*}
Then, recalling again that $\mathrm{div}V=0$ in a neighborhood of $\partial B_1$, the boundary integral in \eqref{dersec(1)} can be rewritten as follows
\begin{equation}\label{parte2}
    \begin{split}
        \int_{\partial\Omega(t)}\mathrm{div}(\abs{\nabla u_t}\abs{\nabla p_t}V)(V\cdot\nu_t)\,d\mathcal{H}^{n-1}&=\int_{\partial\Omega(t)}\abs{\nabla p_t}\frac{\langle D^2u_t.\nabla u_t, V\rangle}{\abs{\nabla u_t}}(V\cdot\nu_t)\,d\mathcal{H}^{n-1}\\
        &+\int_{\partial\Omega(t)}\abs{\nabla u_t}\frac{\langle D^2p_t.\nabla p_t, V\rangle}{\abs{\nabla p_t}}(V\cdot\nu_t)\,d\mathcal{H}^{n-1}.
    \end{split}
\end{equation}
Finally, combining \eqref{parte1}-\eqref{parte2}, we have that
\begin{equation*}
    \begin{split}
        f''(t)=&-\int_{\partial\Omega(t)}\left[\abs{\nabla p_t}\frac{\partial\dot{u}_t}{\partial\nu_t}+\abs{\nabla u_t}\frac{\partial\dot{p}_t}{\partial\nu_t} \right](V\cdot\nu_t)\,d\mathcal{H}^{n-1}+\int_{\partial\Omega(t)}\abs{\nabla p_t}\frac{\langle D^2u_t.\nabla u_t, V\rangle}{\abs{\nabla u_t}}(V\cdot\nu_t)\,d\mathcal{H}^{n-1}\\
        &+\int_{\partial\Omega(t)}\abs{\nabla u_t}\frac{\langle D^2p_t.\nabla p_t, V\rangle}{\abs{\nabla p_t}}(V\cdot\nu_t)\,d\mathcal{H}^{n-1}.
    \end{split}
\end{equation*}
We remark that, from \eqref{Problema(p)}, $\dot{p}_t$ satisfies
\begin{equation}\label{Problema(p.)}
    \begin{cases}
        -\Delta \dot{p}_t=j''(u_t)\dot{u}_t & \text{ in }\Omega(t)\\
        \dot{p}_t=-\nabla p_t\cdot V & \text{ on } \partial\Omega(t).
    \end{cases}
\end{equation}
Now we decompose $V=V^\tau+(V\cdot\nu_t)\nu_t$, and, using the boundary condition in \eqref{Problema(u.)}-\eqref{Problema(p.)}, we find the following
\begin{equation}\label{dersec(2)}
    \begin{split}
        f''(t)=&-\int_{\partial\Omega(t)}\left[\dot{p}_t\frac{\partial\dot{u}_t}{\partial\nu_t}+\dot{u}_t\frac{\partial\dot{p}_t}{\partial\nu_t} \right]\,d\mathcal{H}^{n-1}-\int_{\partial\Omega(t)}\abs{\nabla p_t}\langle D^2u_t.\nu_t,V^\tau\rangle(V\cdot\nu_t)\,d\mathcal{H}^{n-1}\\
        &-\int_{\partial\Omega(t)}\abs{\nabla p_t}\langle D^2u_t.\nu_t,\nu_t\rangle(V\cdot\nu_t)^2\,d\mathcal{H}^{n-1}-\int_{\partial\Omega(t)}\abs{\nabla u_t}\langle D^2p_t.\nu_t,V^\tau\rangle(V\cdot\nu_t)\,d\mathcal{H}^{n-1}\\
        &-\int_{\partial\Omega(t)}\abs{\nabla u_t}\langle D^2p_t.\nu_t,\nu_t\rangle(V\cdot\nu_t)^2\,d\mathcal{H}^{n-1}.
    \end{split}
\end{equation}
We remark that, whenever $g$ is a $C^2$ concave functions that vanish on the boundary of $\Omega(t)$, we have
\begin{equation}\label{infLapl}
    \langle D^2g.\nu_t,\nu_t\rangle=\Delta g+\abs{\nabla g}\mathscr{H}_{\partial\Omega(t)},
\end{equation}
where $\mathscr{H}_{\partial\Omega(t)}$ is the mean curvature of the boundary of $\Omega(t)$. Then, from \eqref{Problema(u)}-\eqref{Problema(p)}-\eqref{dersec(2)}-\eqref{infLapl}, we find
\begin{equation*}
    \begin{split}
        f''(t)=&-\int_{\partial\Omega(t)}\left[\dot{p}_t\frac{\partial\dot{u}_t}{\partial\nu_t}+\dot{u}_t\frac{\partial\dot{p}_t}{\partial\nu_t} \right]\,d\mathcal{H}^{n-1}-\int_{\partial\Omega(t)}\abs{\nabla p_t}\langle D^2u_t.\nu_t,V^\tau\rangle(V\cdot\nu_t)\,d\mathcal{H}^{n-1}\\
        &-\int_{\partial\Omega(t)}\abs{\nabla u_t}\langle D^2p_t.\nu_t,V^\tau\rangle(V\cdot\nu_t)\,d\mathcal{H}^{n-1}+\int_{\partial\Omega(t)}\abs{\nabla p_t}(V\cdot\nu_t)^2\,d\mathcal{H}^{n-1}\\
        &+\int_{\partial\Omega(t)}\abs{\nabla u_t}j'(u_t)(V\cdot\nu_t)^2\,d\mathcal{H}^{n-1}-2\int_{\partial\Omega(t)}\abs{\nabla u_t}\abs{\nabla p_t}\mathscr{H}_{\partial\Omega(t)}(V\cdot\nu_t)^2\,d\mathcal{H}^{n-1}.
    \end{split}
\end{equation*}

Finally, using the divergence Theorem on the first integral, since $u_t=p_t=0$ on $\partial\Omega(t)$, from \eqref{Problema(u.)}-\eqref{Problema(p.)}, we can rewrite $f''(t)$ as follows
\begin{equation*}
    \begin{split}
        f''(t)=&-2\int_{\Omega(t)}\nabla \dot{u}_t\cdot\nabla\dot{p}_t \,dx+\int_{\Omega(t)}j''(u_t)\dot{u}_t^2\,dx+\int_{\partial\Omega(t)}\frac{\abs{\nabla p_t}}{\abs{\nabla u_t}}\langle D^2u_t.\nabla u_t,V^\tau\rangle(V\cdot\nu_t)\,d\mathcal{H}^{n-1}\\
        &+\int_{\partial\Omega(t)}\frac{\abs{\nabla u_t}}{\abs{\nabla p_t}}\langle D^2p_t.\nabla p_t,V^\tau\rangle(V\cdot\nu_t)\,d\mathcal{H}^{n-1}+\int_{\partial\Omega(t)}\abs{\nabla p_t}(V\cdot\nu_t)^2\,d\mathcal{H}^{n-1}\\
        &+\int_{\partial\Omega(t)}\abs{\nabla u_t}j'(u_t)(V\cdot\nu_t)^2\,d\mathcal{H}^{n-1}-2\int_{\partial\Omega(t)}\abs{\nabla u_t}\abs{\nabla p_t}\mathscr{H}_{\partial\Omega(t)}(V\cdot\nu_t)^2\,d\mathcal{H}^{n-1}.
    \end{split}
\end{equation*}
\end{proof}
\section{Stability for nearly spherical sets}\label{stability}
In order to study the stability of critical shapes, one cannot rely on a positivity condition for the full second shape derivative with respect to arbitrary deformation fields. Indeed, shape functionals are invariant under reparametrizations of the domain: tangential deformations, or more generally diffeomorphisms leaving the set unchanged, do not modify the shape and therefore necessarily generate degenerate directions.

As in the first-order Hadamard structure theorem (see \cite{DL}), the relevant information is carried by the normal component of the perturbation. At a critical shape, the second variation reduces to a quadratic form acting on normal boundary deformations. Hence, in the class of nearly spherical sets, stability is obtained by proving coercivity of this quadratic form on the admissible perturbations, after removing the neutral directions associated with the invariances of the problem.

This coercivity estimate must be combined with an improved continuity property of the second variation. The latter ensures that the quadratic expansion remains uniform along nearly spherical perturbations and that the nonlinear remainder does not destroy the positivity given by the second variation. These two ingredients, improved continuity and coercivity, are therefore the key tools for deriving local stability results.

Hence, our first step is to evaluate $f''$ at $t=0$. In order to do this, we recall that $\dot{u}_0$ is exactly the harmonic extension $H(\nabla u_0\cdot V)$. Then our aim is to write $\dot{p}_0$, that is not harmonic, in terms of an harmonic extension. To this aim, we introduce the following Dirichlet problem
\begin{equation}\label{Problema(Q)}
    \begin{cases}
        -\Delta Q=j''(u_0)H(V\cdot\nu_0) & \text{ in } B_1\\
        Q=0 & \text{ on } \partial B_1.
    \end{cases}
\end{equation}
Then, since the harmonic extension satisfies
\begin{equation}\label{Problema(H)}
    \begin{cases}
        -\Delta H(V\cdot\nu_0)=0 & \text{ in } B_1\\
        H(V\cdot\nu_0)=\abs{\nabla u_0}V\cdot\nu_0 & \text{ on } \partial B_1,
    \end{cases}
\end{equation}
from \eqref{Problema(Q)}-\eqref{Problema(H)} we have that
\begin{equation*}
    \dot{p}_0=\abs{\nabla p_0}H(V\cdot\nu_0)+\abs{\nabla u_0}Q.
\end{equation*}
The last gives us
\begin{equation}\label{Grad(p.)}
    \nabla \dot{p}_0=\abs{\nabla p_0}\nabla H(V\cdot\nu_0)+\abs{\nabla u_0}\nabla Q.
\end{equation}
Then, from \eqref{dersec(3)}-\eqref{Grad(p.)}, since $V^\tau|_{\partial B_1}=0$, $\mathscr{H}_{\partial B_1}=n-1$, we have
\begin{equation*}
    \begin{split}
        f''(0)=&-\abs{\nabla u_0}\abs{\nabla p_0}\left\{2\int_{B_1}\abs{\nabla H(V\cdot\nu_0)}^2\,dx+\left[2(n-1)-\frac{1}{\abs{\nabla u_0}}-\frac{j'(0)}{\abs{\nabla p_0}} \right]\int_{\partial B_1}(V\cdot \nu_0)^2\,d\mathcal{H}^{n-1} \right\}\\
        &+\abs{\nabla u_0}^2\int_{B_1}j''(u_0)H(V\cdot\nu_0)^2\,dx.
    \end{split}
\end{equation*}
\subsection{Improved Continuity}
Our next step is to prove the \emph{improved continuity} of our functional, to do this we follow the guide of \cite[Lemma A.2]{BDPV}. This is the content of the next result.
\begin{lemma}[Improved Continuity]
    Let $\gamma\in(0,1]$, there exist $\delta_a=\delta_a(n,\gamma)$ and a modulus of continuity $\omega_a$ such that if $\Omega$, $\varphi$, $V$ and $\Phi_t$ are as in Lemma \ref{lemmaA.1} and $\|\varphi\|_{C^{2,\gamma}}\leq \delta_a$, then
    \begin{equation*}
        |f''(t)-f''(0)|\leq \omega_a(\|\varphi\|_{C^{2,\gamma}})\|V\cdot\nu_0\|_{H^{1/2}(\partial B_1)}^2.
    \end{equation*}
\end{lemma}
\begin{proof}
    We start from \eqref{dersec(3)} and pull it back on $B_1$ through $\Phi_t$:
    \begin{equation*}
        \begin{split}
            f''(t)=&-2\int_{B_1}\nabla\dot{u}_t\cdot\nabla\dot{p}_t\,\circ\Phi_t\,\det{\nabla\Phi_t}\,dx+\int_{B_1}j''(u_t)\dot{u}_t^2\,\circ\Phi_t\,\det{\nabla\Phi_t}\,dx\\
            &+\int_{\partial B_1}\abs{\nabla p_t}(V\cdot\nu_t)^2\,\circ\Phi_t\,J^{\partial B_1}\Phi_t\,d\mathcal{H}^{n-1}+\int_{\partial B_1}\abs{\nabla u_t}j'(u_t)(V\cdot\nu_t)^2\,\circ\Phi_t\,J^{\partial B_1}\Phi_t\,d\mathcal{H}^{n-1}\\
            &+\int_{\partial B_1}\frac{\abs{\nabla p_t}}{\abs{\nabla u_t}}\langle D^2u_t.\nabla u_t,V^\tau\rangle(V\cdot\nu_t)\,\circ\Phi_t\,J^{\partial B_1}\Phi_t\,d\mathcal{H}^{n-1}\\
            &+\int_{\partial B_1}\frac{\abs{\nabla u_t}}{\abs{\nabla p_t}}\langle D^2p_t.\nabla p_t,V^\tau\rangle(V\cdot\nu_t)\,\circ\Phi_t\,J^{\partial B_1}\Phi_t\,d\mathcal{H}^{n-1}\\
            &-2\int_{\partial B_1}\abs{\nabla u_t}\abs{\nabla p_t}\mathscr{H}_{\partial\Omega(t)}(V\cdot\nu_t)^2\,\circ\Phi_t\,J^{\partial B_1}\Phi_t\,d\mathcal{H}^{n-1}\\
            =&I_1(t)+I_2(t)+I_3(t)+I_4(t)+I_5(t)+I_6(t)+I_7(t).
        \end{split}
    \end{equation*}
    By Lemma \ref{lemmaA.1} we get
    \begin{equation}\label{StimeL^inf}
        \|\mathscr{H}_{\partial\Omega(t)}-\mathscr{H}_{\partial B_1}\|_{L^\infty(\partial B_1)}+\|J^{\partial B_1}\Phi_t-1\|_{L^\infty(\partial B_1)}+\|\det{\nabla \Phi_t}-1\|_{L^\infty(B_1)}\leq\omega(\|\varphi\|_{C^{2,\gamma}}),
    \end{equation}
    where in the last equation and in the following $\omega$ denotes a modulus of continuity that will change during the proof.
    
    If we consider the equation satisfied by $u_t\circ\Phi_t$ and $p_t\circ\Phi_t$ in $B_1$, \emph{Schauder estimates} (see for instance \cite[Section 6.3]{GT}) give us
    \begin{equation}\label{Schauder}
        \|u_0-u_t\circ\Phi_t\|_{C^{2,\gamma}(\overline{B}_1)}\leq\omega(\|\varphi\|_{C^{2,\gamma}}),\quad \|p_0-p_t\circ\Phi_t\|_{C^{2,\gamma}(\overline{B}_1)}\leq\omega(\|\varphi\|_{C^{2,\gamma}}).
    \end{equation}

    From the estimate on the Jacobian in \eqref{StimeL^inf} and from \eqref{Schauder} we get
    \begin{equation}\label{I3}
        \begin{split}
            |I_3(t)-I_3(0)|&=\left|\int_{\partial B_1} \left[\abs{\nabla p_t}(V\cdot\nu_t)^2\,\circ\Phi_t\,J^{\partial B_1}\,\Phi_t-\abs{\nabla p_0}(V\cdot\nu_0)^2 \right]\,d\mathcal{H}^{n-1} \right|\\
            &\leq\left|\int_{\partial B_1}\abs{\nabla p_t}(V\cdot\nu_t)^2\,\circ\Phi_t\left(J^{\partial B_1}\,\Phi_t-1 \right)\,d\mathcal{H}^{n-1}\right|\\
            &+\left|\int_{\partial B_1}\left[\abs{\nabla p_t}(V\cdot\nu_t)^2\,\circ\Phi_t -\abs{\nabla p_0}(V\cdot\nu_0)^2\right]\,d\mathcal{H}^{n-1} \right| \\
            &\leq \omega(\|\varphi\|_{C^{2,\gamma}})\|V\cdot\nu_0\|_{H^{1/2}(\partial B_1)}^2,
        \end{split}
    \end{equation}
    where we have applied
    \begin{equation*}
        \left|(V\cdot\nu_t)\circ\Phi_t-V\cdot\nu_0 \right|\leq \omega(\|\varphi\|_{C^{2,\gamma}})\abs{V\cdot\nu_0},
    \end{equation*}
    which is true from Lemma \ref{lemmaA.1} (i). In the same spirit of \eqref{I3}, since $j'(u_t)=-\Delta p_t$ from \eqref{Problema(p)}, we get
    \begin{equation*}
        \abs{I_4(t)-I_4(0)}\leq \omega(\|\varphi\|_{C^{2,\gamma}})\|V\cdot\nu_0\|_{H^{1/2}(\partial B_1)}^2.
    \end{equation*}
    The same kind of estimate, where in addition we need the control on the mean curvature in \eqref{StimeL^inf}, gives us the following
    \begin{equation*}
        \abs{I_7(t)-I_7(0)}\leq \omega(\|\varphi\|_{C^{2,\gamma}})\|V\cdot\nu_0\|_{H^{1/2}(\partial B_1)}^2.
    \end{equation*}
    Now we want to estimate $|I_5(t)|$, $|I_6(t)|$, where we remark that $I_5(0)=I_6(0)=0$ since $V^\tau|_{\partial B_1}=0$. Again, from Lemma \ref{lemmaA.1}, we get
    \begin{equation}\label{StimaVTan}
        \abs{V^\tau\circ\Phi_t}\leq\omega(\|\varphi\|_{C^{2,\gamma}})\abs{V\cdot\nu_0}.
    \end{equation}
    Then, by using the \emph{Cauchy-Schwarz} inequality, combining \eqref{StimeL^inf}-\eqref{StimaVTan}, we find
    \begin{equation*}
        \abs{I_5(t)}\leq\int_{\partial B_1}\abs{\nabla p}\abs{D^2u}\abs{(V\cdot\nu_t)\circ\Phi_t}\abs{V^\tau\circ\Phi_t}J^{\partial B_1}\Phi_t\,d\mathcal{H}^{n-1}\leq \omega(\|\varphi\|_{C^{2,\gamma}})\|V\cdot\nu_0\|_{H^{1/2}(\partial B_1)}^2.
    \end{equation*}
    The same computations give
    \begin{equation*}
        \abs{I_6(t)}\leq \omega(\|\varphi\|_{C^{2,\gamma}})\|V\cdot\nu_0\|_{H^{1/2}(\partial B_1)}^2.
    \end{equation*}
    Now we want to estimate the volume integrals. Lets start defining $v_t=\dot{u}_t\circ\Phi_t$. Then
    \begin{equation*}
        \nabla v_t=(\nabla\Phi_t)^T\nabla\dot{u}_t\circ\Phi_t,
    \end{equation*}
    where with the superscript $^T$ we denote the transposition of a matrix. From \eqref{Problema(u.)} we have that $v_t$ solves the following linear elliptic problem
    \begin{equation}\label{Problema(v)}
        \begin{cases}
            \mathrm{div}(M_t\nabla v_t)=0 & \text{ in } B_1\\
            v_t=-(\nabla u_t\cdot V)\circ\Phi_t & \text{ on } \partial B_1,
        \end{cases}
    \end{equation}
    where 
    \begin{equation*}
        M_t=\det \nabla\Phi_t(\nabla\Phi_t)^{-T}(\nabla\Phi_t)^{-1}
    \end{equation*}
    is symmetric and positive definite and we have used the notation $A^{-T}=(A^{-1})^T$. Since we are dealing with the same problem as \cite{BDPV}, we report without the proof the following estimate
    \begin{equation}\label{StimaGradv}
        \|\nabla v_t-\nabla\dot{u}_0\|_{L^2(B_1)}\leq\omega(\|\varphi\|_{C^{2,\gamma}})\|V\cdot\nu_0\|_{H^{1/2}(\partial B_1)},
    \end{equation}
    that is the content of \cite[Formula A.24]{BDPV}, in order to simplify the proof.
    
    We are now ready to estimate the volume integral $I_2$ as follows
    \begin{equation*}
        \begin{split}
            \abs{I_2(t)-I_2(0)}&\leq \abs{\int_{B_1}[\dot{u}_t^2j''(u_t)\circ\Phi_t-\dot{u}_0^2j''(u_0)]\,dx}\\
            &+\|\det\nabla\Phi_t-1\|_{L^\infty(B_1)}\abs{\int_{B_1}[\dot{u}_t^2j''(u_t)\circ\Phi_t-\dot{u}_0^2j''(u_0)]\,dx}\\
            &+\|\det\nabla\Phi_t-1\|_{L^\infty(B_1)}\abs{\int_{B_1}\dot{u}_0^2j''(u_0)\,dx}.
        \end{split}
    \end{equation*}
    Then, using the \emph{Poincaré} inequality for zero-mean functions, we find
    \begin{equation*}
        \int_{B_1}\dot{u}_0^2j''(u_0)\,dx\leq\frac{\|j''\|_{L^\infty(B_1)}}{\mu_1(B_1)^2}\|\nabla \dot{u}_0\|_{L^2(B_1)}^2\leq \frac{\|j''\|_{L^\infty(B_1)}}{\mu_1(B_1)^2}\|V\cdot\nu_0\|_{H^{1/2}(\partial B_1)}^2,
    \end{equation*}
    where $\mu_1(\cdot)$ is the first non-trivial \emph{Neumann eigenvalue}. Lastly, we have
    \begin{equation*}
        \begin{split}
            \abs{\int_{B_1}[\dot{u}_t^2j''(u_t)\circ\Phi_t-\dot{u}_0^2j''(u_0)]\,dx}\leq& \|j''\|_{L^\infty(B_1)}\int_{B_1}\abs{v_t^2-\dot{u}_0^2}\,dx\\
            +&\|j''(u_t)\circ\Phi_t-j''(u_0)\|_{L^\infty(B_1)}\int_{B_1}\abs{\dot{u}_0}^2\,dx.
        \end{split}
    \end{equation*}
    The second integral can be estimated by the Poincaré inequality and \eqref{StimeL^inf}. Then, to conclude the estimate $$|I_2(t)-I_2(0)|\leq\omega(\|\varphi\|_{C^{2,\gamma}})\|V\cdot\nu_0\|_{H^{1/2}(\partial B_1)}^2,$$ we remark that from \eqref{StimaGradv}, and again by the Poincaré inequality, follow that
    \begin{equation*}
        \int_{B_1}\abs{v_t^2-\dot{u}_0^2}\,dx\leq\omega(\|\varphi\|_{C^{2,\gamma}})\|V\cdot\nu_0\|_{H^{1/2}(\partial B_1)}^2.
    \end{equation*}
    Our last step is the estimate of $I_1$. As before, we introduce $w_t=\dot{p}_t\circ\Phi_t$, that, from \eqref{Problema(p.)}, solves the following linear elliptic problem
    \begin{equation}\label{Problema(w)}
        \begin{cases}
            \mathrm{div}(M_t\nabla w_t)=j''(u_t)\dot{u}_t\circ\Phi_t\,\det\nabla\Phi_t & \text{ in } B_1\\
            w_t=-(\nabla p_t\cdot V)\circ\Phi_t & \text{ on } \partial B_1.
        \end{cases}
    \end{equation}
    Then, from \eqref{Problema(v)}-\eqref{Problema(w)}, and by triangular inequality, we find
    \begin{equation*}
        \begin{split}
            |I_1(t)-I_1(0)|&\leq2\abs{\int_{B_1}[\nabla v_t\cdot\nabla w_t-\nabla\dot{u}_0\cdot\nabla\dot{p}_0]\,dx}\\
            &+2\abs{\int_{B_1}\left[\left((\nabla\Phi_t)^{-T}\nabla v_t \right)\cdot\left((\nabla\Phi_t)^{-T}\nabla w_t \right)\det\nabla\Phi_t-\nabla v_t\cdot\nabla w_t \right]\,dx}\\
            &=2\abs{\int_{B_1}[\nabla v_t\cdot\nabla w_t-\nabla\dot{u}_0\cdot\nabla\dot{p}_0]\,dx}\\
            &+2\abs{\int_{B_1}\left[\det\nabla\Phi_t(\nabla\Phi_t)^{-1}(\nabla\Phi_t)^{-T}-\mathbb{I}_n \right]\nabla v_t\cdot\nabla w_t\,dx},
        \end{split}
    \end{equation*}
    where in the last equality we have used the following algebraic identities
    \begin{equation*}
        Ax\cdot Ay-x\cdot y=x^TA^TAy-x^Ty=x^T(A^TA-\mathbb{I}_n)y=[(A^TA-\mathbb{I}_n)^Tx]\cdot y=(A^TA-\mathbb{I}_n)x\cdot y.
    \end{equation*}
    Regarding the second term in the previous estimate, we get
    \begin{equation*}
        \begin{split}
            \abs{\int_{B_1}\left[\det\nabla\Phi_t(\nabla\Phi_t)^{-1}(\nabla\Phi_t)^{-T}-\mathbb{I}_n \right]\nabla v_t\cdot\nabla w_t\,dx}&\leq\|\det\nabla\Phi_t(\nabla\Phi_t)^{-1}(\nabla\Phi_t)^{-T}-\mathbb{I}_n\|_{L^\infty(B_1)}\\
            &\times\left[\|\nabla v_t-\nabla\dot{u}_0\|_{L^2(B_1)}+\|\nabla \dot{u}_0\|_{L^2(B_1)} \right]\\
            &\times\left[\|\nabla w_t-\nabla\dot{p}_0\|_{L^2(B_1)}+\|\nabla \dot{p}_0\|_{L^2(B_1)} \right],
        \end{split}
    \end{equation*}
    where we have to estimate just the last row.
    Defining $h=\dot{p}_0-H(\nabla p_0\cdot V)$, we have that $h$ solves
    \begin{equation*}
        \begin{cases}
            -\Delta h=j''(u_0)\dot{u}_0 & \text{ in } B_1\\
            h=0 & \text{ on } \partial B_1.
        \end{cases}
    \end{equation*}
    We remark that, using the divergence Theorem, H\"older inequality and the Poincaré inequality for zero-trace functions, we get
    \begin{equation*}
        \|\nabla h\|_{L^2(B_1)}^2=\int_{B_1}hj''(u_0)\dot{u}_0\,dx\leq\|h\|_{L^2(B_1)}\|j''(u_0)\dot{u}_0\|_{L^2(B_1)}\leq\lambda(B_1)\|\nabla h\|_{L^2(B_1)}\|j''(u_0)\dot{u}_0\|_{L^2(B_1)}
    \end{equation*}
    Therefore, the previous estimate and the Poincaré inequality for zero-mean value functions yields 
    \begin{equation*}
        \|\nabla h\|_{L^2(B_1)}\leq \lambda_1(B_1)\|j''(u_0)\dot{u}_0\|_{L^2(B_1)}\leq \frac{\lambda_1(B_1)}{\mu_1(B_1)}\|j''\|_{L^\infty(B_1)}\|\nabla \dot{u}_0\|_{L^2(B_1)},
    \end{equation*}
    where $\lambda_1(\cdot)$ is the first Dirichlet eigenvalue of the Laplacian. From the latter, using the reverse triangle inequality and absorbing all constants into the modulus of continuity, we obtain
    \begin{equation*}
        \|\nabla\dot{p}_0\|_{L^2(B_1)}\leq\omega(\|\varphi\|_{C^{2,\gamma}})\|V\cdot\nu_0\|_{H^{1/2}(\partial B_1)}.
    \end{equation*}
    Then, the very last step to conclude the proof is the following estimate
    \begin{equation}\label{Stimaw}
        \|\nabla w_t-\nabla \dot{p}_0\|_{L^2(B_1)}\leq \omega(\|\varphi\|_{C^{2,\gamma}})\|V\cdot\nu_0\|_{H^{1/2}(\partial B_1)}.
    \end{equation}
    From classical elliptic estimates we find
    \begin{equation}\label{3}
        \begin{split}
            \|\nabla w_t-\nabla \dot{p}_0\|_{L^2(B_1)}\leq C(n)\{&\|(M_t-\mathbb{I}_n)\nabla w_t\|_{L^{2}(B_1)}+\|j''(u_t)\dot{u}_t\circ\Phi_t-j''(u_0)\dot{u}_0\|_{L^2(B_1)}\\
            &+\|(\nabla p_t\cdot V)\circ\Phi_t-\nabla p_0\cdot V\|_{H^{1/2}(\partial B_1)}\}.
        \end{split}
    \end{equation}
    Then, recalling that
    \begin{equation*}
        \|(\nabla p_t\cdot V)\circ\Phi_t-\nabla p_0\cdot V\|_{H^{1/2}(\partial B_1)}\leq \omega(\|\varphi\|_{C^{2,\gamma}})\|V\cdot\nu_0\|_{H^{1/2}(\partial B_1)},
    \end{equation*}
    since all the terms that appear on the right hand side in \eqref{3} have already been estimated, we find \eqref{Stimaw}.
\end{proof}
\subsection{Coercivity}
Let us now define the bilinear form induced by the second order shape derivative:
\begin{equation}\label{FormaBilineare}
    \begin{split}
    \partial^2 \mathcal{J}(B_1)[\xi,\xi]=&2\abs{\nabla u_0}\abs{\nabla p_0}\left\{\int_{B_1}\abs{\nabla H(\xi)}^2\,dx+\left[n-1-\frac{1}{2\abs{\nabla u_0}}-\frac{j'(0)}{2\abs{\nabla p_0}} \right]\int_{\partial B_1}\xi^2\,d\mathcal{H}^{n-1} \right\}\\
     &-\abs{\nabla u_0}^2\int_{B_1}j''(u_0)H(\xi)^2\,dx.
    \end{split}
\end{equation}
\begin{oss}\label{ossBilineare}
    It holds 
    \begin{equation*}
        \partial^2 \mathcal{J}(B_1)[V\cdot\nu_0,V\cdot\nu_0]=-f''(0)>0,
    \end{equation*}
    since $B_1$ is a maximum for $\mathcal{J}(\cdot)$.
\end{oss}

In the next result we prove the continuity for $\partial^2\mathcal{J}(B_1)[\cdot,\cdot]$, which will be useful in the proof of the coercivity.

\begin{lemma}[Continuity of $\partial^2\mathcal{J}(B_1)$]\label{continuitàforma}

There exist a constant $c_1(n,j)$ such that
\begin{equation*}
    \partial^2\mathcal{J}(B_1)[\xi_1,\xi_2]\leq c_1(n,j)\|\xi_1\|_{H^{1/2}(\partial B_1)}\|\xi_2\|_{H^{1/2}(\partial B_1)}.
\end{equation*}
\end{lemma}
\begin{proof}
    By a triangle inequality we get
    \begin{equation*}
        \begin{split}
            \abs{\partial^2\mathcal{J}(B_1)[\xi_1,\xi_2]}\leq& \frac{2}{n}\abs{\nabla p_0}\int_{B_1}\abs{\nabla H(\xi_1) \cdot \nabla H(\xi_2)}\,dx\\
            +&\frac{2}{n}\abs{\nabla p_0}\abs{\frac{n}{2}-1-\frac{j'(0)}{2\abs{\nabla p_0}}}\int_{\partial B_1}\xi_1\xi_2\,d\mathcal{H}^{n-1}\\
            +&\frac{1}{n^2}\abs{\int_{B_1}j''(u_0)H(\xi_1)H(\xi_2)\,dx}.
        \end{split}
    \end{equation*}
    On the first two term the continuity follows just after a H\"older inequality and the definition of $\|\cdot\|_{H^{1/2}(\partial B_1)}$. Let us now focus on the last term, by boundedness of $j''$ and H\"older inequality we find
    \begin{equation*}
        \abs{\int_{B_1}j''(u_0)H(\xi_1)H(\xi_2)\,dx}\leq\norma{j''}_{L^\infty(B_1)}\norma{H(\xi_1)}_2\norma{H(\xi_2)}_2\leq\frac{\norma{j''}_{L^\infty(B_1)}}{n}\|\xi_1\|_{H^{1/2}(\partial B_1)}\|\xi_2\|_{H^{1/2}(\partial B_1)}.
    \end{equation*}
\end{proof}
\begin{lemma}[Coercivity]\label{Coercività2}
    Let $\partial^2 \mathcal{J}(B_1)$ be the symmetric and continuous bilinear form on $H^{1/2}(\partial B_1)$ defined in \eqref{FormaBilineare}. Then there exists $\delta>0$ depending on $j,n$ and a constant $c_2(n,j,\delta)$ such that 
    \begin{equation*}
        \partial^2\mathcal{J}(B_1)[\xi,\xi]\geq c_2(n,j,\delta)\norma{\xi}_{H^{1/2}(\partial B_1)}^2\quad\text{for every }\xi\in \mathcal{M}_\delta,
    \end{equation*}
    where
        \begin{equation*}
        \mathcal{M}_\delta=\left\{\xi\in H^{1/2}(\partial B_1)\,:\, \abs{\int_{\partial B_1}\xi\,d\mathcal{H}^{n-1}}+\abs{\int_{\partial B_1}x\xi\,d\mathcal{H}^{n-1}}\leq\delta\norma{\xi}_{H^{1/2}(\partial B_1)}\right\}.
    \end{equation*}
\end{lemma}
\begin{proof}
    Let us focus first on proving the estimate in the space $\mathcal{M}_0$ that is defined as
    \begin{equation*}
        \mathcal{M}_0=\left\{\xi\in H^{1/2}(\partial B_1)\,:\, \int_{\partial B_1}\xi\,d\mathcal{H}^{n-1}=\int_{\partial B_1}x_i\xi\,d\mathcal{H}^{n-1}=0,\, i=1,\dots,n \right\}.
    \end{equation*}
    First notice by Green's identity we have that:
    \begin{equation*}
        \int_{B_1}j'(u_0)\,dx=\int_{B_1}-\Delta p_0\,dx=\int_{\partial B_1}-\nd{p_0}\dH=\int_{\partial B_1}\abs{\nabla p_0}=\abs{\nabla p_0}n\omega_n,
    \end{equation*}
    therefore, to simplify the notation, let us set
    \begin{equation*}
    \abs{\nabla p_0}=\frac{\int_{B_1}j'(u_0)\,dx}{n\omega_n}=\gamma_j.
    \end{equation*}
    Moreover, by introducing the basis of spherical harmonics in $\mathcal{M}_0$, as defined before, then each term in \eqref{FormaBilineare} becomes
    \begin{equation*}
        \begin{split}
            \int_{\partial B_1}&\xi^2\dH=\sum_{k\geq2}c_k^2,\\
            \int_{B_1}&\abs{\nabla H(\xi)}^2\,dx=\sum_{k\geq2}kc_k^2,\\
            \int_{B_1}&j''(u_0)H(\xi)^2\,dx=\sum_{k\geq2}c_k^2
            \int_0^1 j''\left(\frac{1-r^2}{2n}\right)r^{2k+n-1}\,dr=\sum_{k\geq2}c_k^2I_k.
        \end{split}
    \end{equation*}
    Now, let us notice that if $r\in[0,1]$ we have that $r^{2(k+1)+n-1}\leq r^{2k+n-1}$ which implies that $I_k$ is non increasing and it holds
    \begin{equation*}
        I_k\leq\norma{j''}_{L^\infty}\int_0^1r^{2k+n-1}\,dr=\frac{\norma{j''}_{L^\infty}}{2k+n}\rightarrow0\quad\text{for }k\rightarrow\infty,
    \end{equation*}
    where the norm of $j''$ doesn't blow up due to the assumption.
    
    Hence, \eqref{FormaBilineare} becomes
    \begin{equation*}
        \partial^2 \mathcal{J}(B_1)[\xi,\xi]=\sum_{k\geq2}c_k^2\left[\frac{2\gamma_jk}{n}+\frac{\gamma_j(n-2)-j'(0)}{n}-\frac{I_k}{n^2}\right]=\sum_{k\geq2}c_k^2A_k,
    \end{equation*}
    thus the proof concludes if we find a positive constant $C$ such that $A_k\geq Ck$ for every $k\geq2$. Now, we have that 
    \begin{equation*}
        A_{k+1}-A_k=\frac{2\gamma_j}{n}-\frac{I_{k+1}-I_k}{n^2}>\frac{2\gamma_j}{n}>0
    \end{equation*}
    which means that $A_k$ is increasing and $A_k>A_2$. Moreover, $A_2$ can be rewritten using an integration by parts as
    \begin{equation*}
        A_2=\frac{n+2}{n}\left[\gamma_j-\int_0^1j'\left(\frac{1-r^2}{2n}\right)r^{n+1}\,dr\right].
    \end{equation*}
    Also, using coarea formula we have
    \begin{equation*}
        \frac{1}{n\omega_n}\int_{B_1}j'(u_0)\abs{x}^2\,dx=\frac{1}{n\omega_n}\int_0^1j'\left(\frac{1-r^2}{2n}\right)r^2\int_{\partial B_r}\,\dH\,dr=\int_0^1j'\left(\frac{1-r^2}{2n}\right)r^{n+1}\,dr,
    \end{equation*}
    hence, using definition of $\gamma_j$, the bracket in $A_2$ is given by
    \begin{equation*}
        \frac{1}{n\omega_n}\int_{B_1} j'(u_0)-j'(u_0)\abs{x}^2\,dx=\frac{1}{n\omega_n}\int_{B_1} j'(u_0)(1-\abs{x}^2)\,dx,
    \end{equation*}
    thus $A_2$ is positive if $j'(u_0)$ is positive.

    Therefore, if we consider the sequence $\frac{A_k}{k}$, it is positive and converges to $\frac{2\gamma_j}{n}$, so it admits a positive minimum $\overline{C}(n,j)$. Thus, using the equivalence gives
    \begin{equation*}
        \partial^2 \mathcal{J}(B_1)[\xi,\xi]=\sum_{k\geq2}c_k^2A_k\geq \overline{C}\sum_{k\geq2}kc_k^2\geq\frac{\overline{C}}{2}\norma{\xi}^2_{H^{1/2}(\partial B_1)}.
    \end{equation*}
    For every $\xi\in\mathcal{M}_\delta$ let us consider its $L^2$ projection on $\mathcal{M}_0^\perp$, given by
    \begin{equation*}
        \Pi(\xi)=a_0W_0+\sum_{i=1}^na_{1,i}W_{1,i},
    \end{equation*}
    where
    \begin{equation*}
        W_0=\sqrt{\frac{1}{nw_n}}\qquad W_{1,i}(x)=\sqrt{\frac{1}{w_n}}x_i,\quad x\in\partial B_1,\quad i=1,\dots,n
    \end{equation*}
    and
    \begin{equation*}
        a_0=\int_{\partial B_1}\xi W_0\dH\qquad a_{1,i}=\int_{\partial B_1}\xi W_{1,i}\dH,\quad i=1,\dots,n
    \end{equation*}
    then $\xi-\Pi(\xi)\in\mathcal{M}_0$. Moreover, by Green formula and definition of $\mathcal{M}_\delta$
    \begin{equation}\label{diffxiproiezione}
        \norma{\xi-\Pi(\xi)}^2_{H^{1/2}(\partial B_1)}=\norma{\xi}^2_{H^{1/2}(\partial B_1)}-\norma{\Pi(\xi)}^2_{H^{1/2}(\partial B_1)}\geq\norma{\xi}^2_{H^{1/2}(\partial B_1)}-C\delta^2\norma{\xi}^2_{H^{1/2}(\partial B_1)}.
    \end{equation}
    Hence using bilinearity, the identity \eqref{diffxiproiezione}, the previous computation in $\mathcal{M}_0$ and the continuity \ref{continuitàforma} gives
    \begin{equation*}
        \begin{split}
            \partial^2\mathcal{J}(B_1)[\xi,\xi]&=\partial^2\mathcal{J}(B_1)[\xi-\Pi(\xi),\xi-\Pi(\xi)]+2\partial^2\mathcal{J}(B_1)[\xi,\Pi(\xi)]-\partial^2\mathcal{J}(B_1)[\Pi(\xi),\Pi(\xi)]\geq\\
            &\geq \frac{\overline{C}}{2}\norma{\xi-\Pi(\xi)}^2_{H^{1/2}(\partial B_1)}-2C\norma{\xi}_{H^{1/2}(\partial B_1)}\norma{\Pi(\xi)}_{H^{1/2}(\partial B_1)}-C\norma{\Pi(\xi)}^2_{H^{1/2}(\partial B_1)}\geq\\
            &\geq\frac{\overline{C}}{2}\norma{\xi}^2_{H^{1/2}(\partial B_1)}-C\delta\norma{\xi}^2_{H^{1/2}(\partial B_1)}
        \end{split}
    \end{equation*}
    which gives the thesis for $\delta$ sufficiently small.
\end{proof}

We conclude by adding that, furthermore, in Appendix \ref{alternative} we provide a set of different assumptions regarding $j$ for which the result remains valid.
\subsection{Proof of the stability for nearly spherical set}
The last result that we need to conclude the local stability, is the following lemma due to Dambrine, see \cite[Theorem 1]{D}.
\begin{lemma}\label{lemma 3.4}
    Let $0<\gamma\leq1$, there exist a modulus of continuity $\omega$ and a constant $\delta_b=\delta_b(n,\gamma)$, such that, for every $C^{2,\gamma}$ nearly spherical set $\Omega$ parametrized by $\varphi$ with $\|\varphi\|_{C^{2,\gamma}}\leq\delta_b$ and $\abs{\Omega}=\abs{B_1}$, we have
    \begin{equation*}
        \mathcal{J}(\Omega)\leq\mathcal{J}(B_1)-\frac{1}{2}\partial^2\mathcal{J}(B_1)[\varphi,\varphi]+\omega(\|\varphi\|_{C^{2,\gamma}})\|\varphi\|_{H^{1/2}(\partial B_1)}^2.
    \end{equation*}
\end{lemma}
\begin{proof}
    By Taylor formula we have 
    \begin{equation*}
        \mathcal{J}(\Omega)=\mathcal{J}(B_1)+f'(0)+\frac{1}{2}f''(0)+\int_0^1(1-s)(f''(s)-f''(0))\,ds.
    \end{equation*}
    Then from Talenti we know that $f'(0)=0$ and from Remark \ref{ossBilineare} we have $f''(0)=-\partial^2\mathcal{J}(B_1)[V\cdot\nu_0,V\cdot\nu_0]$. Then the previous identity becomes
    \begin{equation}\label{1}
        \mathcal{J}(\Omega)=\mathcal{J}(B_1)-\frac{1}{2}\partial^2\mathcal{J}(B_1)[V\cdot\nu_0,V\cdot\nu_0]+\int_0^1(1-s)(f''(s)-f''(0))\,ds.
    \end{equation}
    Since $\partial^2\mathcal{J}(B_1)[\cdot,\cdot]$ is symmetric and continuous, from Lemma \ref{lemmaA.1}, we find
    \begin{equation}\label{2}
        \abs{\partial^2\mathcal{J}(B_1)[V\cdot\nu_0,V\cdot\nu_0]-\partial^2\mathcal{J}(B_1)[\varphi,\varphi]}\leq\omega(\|\varphi\|_{L^\infty(\partial B_1)})\|\varphi\|_{H^{1/2}(\partial B_1)}^2.
    \end{equation}
    Then, combining \eqref{1}-\eqref{2}, we get the thesis.
\end{proof}
We now have all the ingredients to prove stability for nearly spherical sets.

\begin{proof}[Proof of Theorem \ref{MainTh}]
    By assumption on $\varphi$ we get
    \begin{equation*}
        \abs{\Omega}=\int_{\partial B_1}\frac{(1+\varphi)^n}{n}\,d\mathcal{H}^{n-1},\quad x_\Omega=\int_{\partial B_1}y\frac{(1+\varphi)^{n+1}}{n+1}\,d\mathcal{H}^{n-1}=0.
    \end{equation*}
    Thanks to the smallness assumption on $\varphi$ we have
    \begin{equation*}
        \abs{\int_{\partial B_1}\varphi\,d\mathcal{H}^{n-1}}=\abs{\sum_{k=2}^{n}\binom{n}{k}\int_{\partial B_1}\frac{\varphi^k}{n}\,d\mathcal{H}^{n-1}}\leq c\int_{\partial B_1}\varphi^2\,d\mathcal{H}^{n-1}\leq c\delta_c\|\varphi\|_{H^{1/2}(\partial B_1)},
    \end{equation*}
    and
    \begin{equation*}
        \abs{\int_{\partial B_1}y_i\varphi\,d\mathcal{H}^{n-1}}\leq\sum_{k=2}^{n+1}\binom{n+1}{k}\int_{\partial B_1}\abs{\frac{\varphi^k}{n+1}}\,d\mathcal{H}^{n-1}\leq c\delta_c\|\varphi\|_{H^{1/2}(\partial B_1)},
    \end{equation*}
    where $c=c(n)$ is a dimensional constant. In the same notation of Lemma \ref{Coercività2}, we have that $\varphi\in\mathcal{M}_{c\delta_c}$ and, by Lemmas \ref{Coercività2}-\ref{lemma 3.4}, if $\delta_c<<\min\{\delta_b, \delta_d\}$, we can infer:
    \begin{equation*}
        \begin{split}
            \mathcal{J}(B_1)-\mathcal{J}(\Omega)&\geq \frac{1}{2}\partial^2\mathcal{J}(B_1)[\varphi,\varphi]-\omega(\|\varphi\|_{C^{2,\gamma}})\|\varphi\|_{H^{1/2}(\partial B_1)}^2\\
            &\geq C(n,j)\|\varphi\|_{H^{1/2}(\partial B_1)}^2.
        \end{split}
    \end{equation*}
\end{proof}

\appendix
\section{Some functionals that satisfy the assumptions}\label{listafunzione}
In the following with $u_\Omega$ we will denote the solution to \eqref{Problema(u)} in $\Omega$, while with
\begin{equation*}
    u_{B_1}(x)=\frac{1-|x|^2}{2n},
\end{equation*}
the solution to \eqref{Problema(u)} in $B_1$.
\begin{itemize}
    \item The \emph{Moser-Trudinger functional} in dimension $n=2$:
    \[ j(s)=e^{s^2}-1.\]
    For this function, the stability result reads as follow
    \begin{equation*}
        \int_{B_1}e^{u_{B_1}(x)^2}\,dx-\int_\Omega e^{u_\Omega(x)^2}\,dx\geq c(n)\|\varphi\|_{H^{1/2}(\partial B_1)}^2.
    \end{equation*}
    We notice that for $n\geq3$, the function $j(s)=e^{s^\frac{n}{n-1}}$ no longer has a bounded second derivative.
    \item The \emph{$p$-th power of the Lebesgue norm} for $2\leq p<\infty$:
    \[j(s)=s^p.\]
    For this function, the stability result reads as follow
    \begin{equation*}
        \|u_{B_1}\|_{L^p(B_1)}^p-\|u_\Omega\|_{L^p(\Omega)}^p\geq c(n)\|\varphi\|_{H^{1/2}(\partial B_1)}^2
    \end{equation*}
    We notice that for $1<p<2$, the function $j(s)=s^p$ no longer has a bounded second derivative.
    \item The \emph{Torsional rigidity} (or $L^1$-norm):
    \[j(s)=s.\]
    In this case our result recovers, in the class of nearly spherical sets, the result proved in \cite{BDPV}, and it reads as follows
    \begin{equation*}
        T(B_1)-T(\Omega)\geq c(n)\|\varphi\|_{H^{1/2}(\partial B_1)}^2.
    \end{equation*}
\end{itemize}

\section{Alternative proof}\label{alternative}
We can also establish the coercivity of the functional $\partial^2\mathcal{J}$ under a structural assumption on $j$, different from the regularity assumption stated in the introduction. First, we provide a different proof of continuity using a weaker summability assumption
\begin{equation}\label{Integrabilita j''}
    j''\in L^{\frac n2}\!\left(0,\frac{1}{2n}\right),
\end{equation}

\begin{lemma}[Continuity of $\partial^2\mathcal{J}$]
There exists a constant $c_1=c_1(n,j)>0$ such that
\[
    \partial^2\mathcal{J}(B_1)[\xi_1,\xi_2]\leq c_1(n,j)\|\xi_1\|_{H^{1/2}(\partial B_81)}\|\xi_2\|_{H^{1/2}(\partial B_1)}
\]
for every $\xi_1,\xi_2\in H^{1/2}(\partial B_1)$.
\end{lemma}

\begin{proof}
The proof is unchanged, except for the estimate of the last term arising in the triangle inequality, where we exploit the integrability assumption \eqref{Integrabilita j''}. Set
\[
q=\frac n2,\qquad r=\frac{n}{n-2}.
\]
Then $1/q+1/r=1$ and $2r=2^\ast$. By H\"older's inequality, we obtain
\[
\begin{aligned}
\left|\int_{B_1} j''(u_0)\,H(\xi_1)\,H(\xi_2)\,dx\right|&\leq \|j''\|_{L^q(B_1)}\|H(\xi_1)H(\xi_2)\|_{L^r(B_1)} \\&\leq \|j''\|_{L^q(B_1)}\|H(\xi_1)\|_{L^{2^*}(B_1)}\|H(\xi_2)\|_{L^{2^*}(B_1)} \\&\leq c(n)\|j''\|_{L^q(B_1)}\|H(\xi_1)\|_{H^1(B_1)}\|H(\xi_2)\|_{H^1(B_1)},
\end{aligned}
\]
where in the last step we used the Sobolev embedding $H^1(B_1)\hookrightarrow L^{2^*}(B_1)$.
\end{proof}

While, the proof of coercivity only differs in the treatment of the class $\mathcal{M}_0$, and yields an explicit coercivity constant assuming that $j''$ is bounded (at least in the interval $\left(0,\frac{1}{2n}\right)$) and satisfies:
\begin{equation}\label{BoundBassoj}
    \dfrac{j'(0)}{\int_0^1j'(s)(1-2ns)^{(n-2)/2}\,ds}\geq n(n-2),
\end{equation}
\begin{equation}\label{BoundAltoj}
    \frac{j'(0)\beta_{n/2,1}^4+4\|j''\|_{L^\infty(0,1/2n)}}{\beta_{n/2,1}^4\int_0^1j'(s)(1-2ns)^{(n-2)/2}\,ds}\leq n^2(n+2),
\end{equation}
here $\beta_{n/2,1}$ denotes the first positive zero of a Bessel function of the first kind.
\begin{lemma}[Coercivity on $\mathcal{M}_0$]
Let $\partial^2\mathcal{J}$ be the symmetric continuous bilinear form on $H^{1/2}(\partial B_1)$ defined in \eqref{FormaBilineare}. Then $\partial^2\mathcal{J}(B_1)$ is coercive on $\mathcal M_0\setminus\{0\}$, that is, there exists a constant $c_2=c_2(n,j)>0$ such that
\[
    \partial^2\mathcal{J}(B_1)[\xi,\xi]\geq c_2(n,j)\|\xi\|_{H^{1/2}(\partial B_1)}^2\qquad\text{for every }\xi\in\mathcal M_0.
\]
More precisely,
\[
    c_2(n,j)=\frac{|\nabla p_0|}{n}\left[\frac n4+\frac12-\frac{j'(0)}{4|\nabla p_0|}-\frac{\|j''\|_{L^\infty(B_1)}}{n|\nabla p_0|\mu_1(B_1)^2}\right],
\]
and $c_2(n,j)>0$ in view of \eqref{BoundAltoj}.
\end{lemma}

\begin{proof}
We first recall that, by Green's formula,
\[
|\nabla p_0|=\frac{\int_{B_1} j'(u_0)\,dx}{n\omega_n}.
\]

Factoring out the common term, the bilinear form can be written as
\[
\begin{aligned}
\partial^2\mathcal{J}(B_1)[\xi,\xi]&=\frac{2}{n}|\nabla p_0|\,\|\nabla H(\xi)\|_{L^2(B_1)}^2\left[1+\left(\frac n2-1-\frac{j'(0)}{2|\nabla p_0|}\right)\frac{\int_{\partial B_1}\xi^2\,d\mathcal H^{n-1}}{\int_{B_1}|\nabla H(\xi)|^2\,dx}\right] \\&\qquad-|\nabla u_0|^2\int_{B_1} j''(u_0)\,H(\xi)^2\,dx.
\end{aligned}
\]

By \eqref{BoundBassoj}, the coefficient of the Rayleigh quotient in brackets is non-positive. Therefore, using
\[
\min\left\{\frac{\int_{B_1}|\nabla H(\xi)|^2\,dx}{\int_{\partial B_1}\xi^2\,d\mathcal H^{n-1}}:\;\xi\in\mathcal M_0\setminus\{0\}\right\}=2,
\]
we infer that
\[
\partial^2\mathcal{J}(B_1)[\xi,\xi]\geq\frac{2}{n}|\nabla p_0|\,\|\nabla H(\xi)\|_{L^2(B_1)}^2\left[\frac n4+\frac12-\frac{j'(0)}{4|\nabla p_0|}\right]-|\nabla u_0|^2\int_{B_1} j''(u_0)\,H(\xi)^2\,dx.
\]

We now estimate the last term by Poincar\'e's inequality:
\[
-|\nabla u_0|^2\int_{B_1} j''(u_0)\,H(\xi)^2\,dx\geq-\frac{\|j''\|_{L^\infty(B_1)}}{n^2\mu_1(B_1)^2}\|\nabla H(\xi)\|_{L^2(B_1)}^2.
\]
Combining the previous inequalities, we obtain
\[
\partial^2\mathcal{J}(B_1)[\xi,\xi]\geq\frac{2}{n}|\nabla p_0|\left[\frac n4+\frac12-\frac{j'(0)}{4|\nabla p_0|}-\frac{\|j''\|_{L^\infty(B_1)}}{2n|\nabla p_0|\mu_1(B_1)^2}\right]\|\nabla H(\xi)\|_{L^2(B_1)}^2.
\]

Finally, since
\[
\|\nabla H(\xi)\|_{L^2(B_1)}^2\leq\|\xi\|_{H^{1/2}(\partial B_1)}^2\leq2\|\nabla H(\xi)\|_{L^2(B_1)}^2,
\]
it follows that
\[
\partial^2\mathcal{J}(B_1)[\xi,\xi]\geq\frac{|\nabla p_0|}{n}\left[\frac n4+\frac12-\frac{j'(0)}{4|\nabla p_0|}-\frac{\|j''\|_{L^\infty(B_1)}}{n|\nabla p_0|\mu_1(B_1)^2}\right]\|\xi\|_{H^{1/2}(\partial B_1)}^2.
\]
This concludes the proof.
\end{proof}
\paragraph{Funding}
The authors were partially supported by Gruppo Nazionale per l’Analisi Matematica, la Probabilità e le loro Applicazioni
(GNAMPA) of Istituto Nazionale di Alta Matematica (INdAM).

The first named author is partially supported by INdAM GNAMPA 2025 Project ``Proprieta’ qualitative e regolarizzanti di equazioni ellittiche e paraboliche'', CUP E5324001950001.

The second named author is partially supported by INdAM GNAMPA 2026 Project ``Processi di diffusione non-lineari: regolarità e classificazione delle soluzioni'', CUP E53C25002010001
\paragraph{Competing Interests}
We declare that we have no financial and personal relationships with other people or organizations.
\bibliographystyle{plain}
\bibliography{biblio}

@book{HP,
  title = {Shape Variation and Optimization: A Geometrical Analysis},
  ISBN = {9783037196786},
  ISSN = {2943-5005},
  url = {http://dx.doi.org/10.4171/178},
  DOI = {10.4171/178},
  journal = {EMS Tracts in Mathematics},
  publisher = {EMS Press},
  author = {Henrot,  Antoine and Pierre,  Michel},
  year = {2018},
  month = feb 
}

@article {BDPV,
    AUTHOR = {Brasco, Lorenzo and De Philippis, Guido and Velichkov,
              Bozhidar},
     TITLE = {Faber-{K}rahn inequalities in sharp quantitative form},
   JOURNAL = {Duke Math. J.},
  FJOURNAL = {Duke Mathematical Journal},
    VOLUME = {164},
      YEAR = {2015},
    NUMBER = {9},
     PAGES = {1777--1831},
      ISSN = {0012-7094,1547-7398},
   MRCLASS = {49R05 (47A75 49Q20)},
  MRNUMBER = {3357184},
MRREVIEWER = {Antoine\ Lemenant},
       DOI = {10.1215/00127094-3120167},
       URL = {https://doi.org/10.1215/00127094-3120167},
}

@article {D,
    AUTHOR = {Dambrine, Marc},
     TITLE = {On variations of the shape {H}essian and sufficient conditions
              for the stability of critical shapes},
   JOURNAL = {RACSAM. Rev. R. Acad. Cienc. Exactas F\'is. Nat. Ser. A Mat.},
  FJOURNAL = {Revista de la Real Academia de Ciencias Exactas, F\'isicas y
              Naturales. Serie A. Matem\'aticas. RACSAM},
    VOLUME = {96},
      YEAR = {2002},
    NUMBER = {1},
     PAGES = {95--121},
      ISSN = {1578-7303,1579-1505},
   MRCLASS = {49Q10 (35J20)},
  MRNUMBER = {1915674},
MRREVIEWER = {Dorin\ Bucur},
}

@article {BDR,
    AUTHOR = {Brasco, Lorenzo and De Philippis, Guido and Ruffini, Berardo},
     TITLE = {Spectral optimization for the {S}tekloff-{L}aplacian: the
              stability issue},
   JOURNAL = {J. Funct. Anal.},
  FJOURNAL = {Journal of Functional Analysis},
    VOLUME = {262},
      YEAR = {2012},
    NUMBER = {11},
     PAGES = {4675--4710},
      ISSN = {0022-1236,1096-0783},
   MRCLASS = {35P15 (35J05 35J20)},
  MRNUMBER = {2913683},
MRREVIEWER = {J.\ B.\ Kennedy},
       DOI = {10.1016/j.jfa.2012.03.017},
       URL = {https://doi.org/10.1016/j.jfa.2012.03.017},
}

@article {ALT,
	AUTHOR = {Alvino, Angelo and Lions, Pierre-Louis and Trombetti, Guido},
	TITLE = {On optimization problems with prescribed rearrangements},
	JOURNAL = {Nonlinear Anal.},
	FJOURNAL = {Nonlinear Analysis. Theory, Methods \& Applications. An
	International Multidisciplinary Journal},
	VOLUME = {13},
	YEAR = {1989},
	NUMBER = {2},
	PAGES = {185--220},
	MRCLASS = {90C48 (26B35 49A27)},
	MRNUMBER = {979040},
	MRREVIEWER = {Bruce D. Craven},
}

@book {GT,
    AUTHOR = {Gilbarg, David and Trudinger, Neil S.},
     TITLE = {Elliptic partial differential equations of second order},
    SERIES = {Classics in Mathematics},
      NOTE = {Reprint of the 1998 edition},
 PUBLISHER = {Springer-Verlag, Berlin},
      YEAR = {2001},
     PAGES = {xiv+517},
      ISBN = {3-540-41160-7},
   MRCLASS = {35-02 (35Jxx)},
  MRNUMBER = {1814364},
}

@book {K,
    AUTHOR = {Kesavan, Srinivasan},
     TITLE = {Symmetrization \& applications},
    SERIES = {Series in Analysis},
    VOLUME = {3},
 PUBLISHER = {World Scientific Publishing Co. Pte. Ltd., Hackensack, NJ},
      YEAR = {2006},
     PAGES = {xii+148},
      ISBN = {981-256-733-X},
   MRCLASS = {35A25 (26D15 35B05 46E15 47F05 47N20 49Q20)},
  MRNUMBER = {2238193},
MRREVIEWER = {Almut\ Burchard},
       DOI = {10.1142/9789812773937},
       URL = {https://doi.org/10.1142/9789812773937},
}

@article {P,
    AUTHOR = {P\'olya, George},
     TITLE = {Torsional rigidity, principal frequency, electrostatic
              capacity and symmetrization},
   JOURNAL = {Quart. Appl. Math.},
  FJOURNAL = {Quarterly of Applied Mathematics},
    VOLUME = {6},
      YEAR = {1948},
     PAGES = {267--277},
      ISSN = {0033-569X,1552-4485},
   MRCLASS = {52.0X},
  MRNUMBER = {26817},
MRREVIEWER = {G.\ H.\ Handelman},
       DOI = {10.1090/qam/26817},
       URL = {https://doi.org/10.1090/qam/26817},
}

@article {T,
    AUTHOR = {Talenti, Giorgio},
     TITLE = {Elliptic equations and rearrangements},
   JOURNAL = {Ann. Scuola Norm. Sup. Pisa Cl. Sci. (4)},
  FJOURNAL = {Annali della Scuola Normale Superiore di Pisa. Classe di
              Scienze. Serie IV},
    VOLUME = {3},
      YEAR = {1976},
    NUMBER = {4},
     PAGES = {697--718},
      ISSN = {0391-173X,2036-2145},
   MRCLASS = {35J35 (35P15)},
  MRNUMBER = {601601},
MRREVIEWER = {Vladimir\ A.\ Kondratiev},
       URL = {http://www.numdam.org/item?id=ASNSP_1976_4_3_4_697_0},
}

@article {DL,
    AUTHOR = {Dambrine, Marc and Lamboley, Jimmy},
     TITLE = {Stability in shape optimization with second variation},
   JOURNAL = {J. Differential Equations},
  FJOURNAL = {Journal of Differential Equations},
    VOLUME = {267},
      YEAR = {2019},
    NUMBER = {5},
     PAGES = {3009--3045},
      ISSN = {0022-0396,1090-2732},
   MRCLASS = {49Q10 (35J20 49K20)},
  MRNUMBER = {3953026},
MRREVIEWER = {Kevin\ Sturm},
       DOI = {10.1016/j.jde.2019.03.033},
       URL = {https://doi.org/10.1016/j.jde.2019.03.033},
}

@article {BFNT,
    AUTHOR = {Bucur, Dorin and Ferone, Vincenzo and Nitsch, Carlo and
              Trombetti, Cristina},
     TITLE = {The quantitative {F}aber-{K}rahn inequality for the {R}obin
              {L}aplacian},
   JOURNAL = {J. Differential Equations},
  FJOURNAL = {Journal of Differential Equations},
    VOLUME = {264},
      YEAR = {2018},
    NUMBER = {7},
     PAGES = {4488--4503},
      ISSN = {0022-0396,1090-2732},
   MRCLASS = {35P15 (35J05 47J30)},
  MRNUMBER = {3758529},
MRREVIEWER = {Srinivasan\ Kesavan},
       DOI = {10.1016/j.jde.2017.12.014},
       URL = {https://doi.org/10.1016/j.jde.2017.12.014},
}

@article{AB2,
  title = {On the stability of the annulus for the torsion of multiply connected domains},
  volume = {65},
  ISSN = {1432-0835},
  url = {http://dx.doi.org/10.1007/s00526-026-03362-w},
  DOI = {10.1007/s00526-026-03362-w},
  number = {6},
  journal = {Calc. Var. Partial Differential Equations},
  publisher = {Springer Science and Business Media LLC},
  author = {Amato,  Vincenzo and Barbato,  Luca},
  year = {2026},
  month = {June} 
}

@article{ABMP,
  title = {The {T}alenti comparison result in a quantitative form},
  ISSN = {0391-173X},
  url = {http://dx.doi.org/10.2422/2036-2145.202404_010},
  DOI = {10.2422/2036-2145.202404_010},
  journal = {Annali Scuola Normale Superiore - Classe di Scienze},
  publisher = {Scuola Normale Superiore - Edizioni della Normale},
  author = {Amato, Vincenzo and Barbato, Rosa and Masiello, Alba Lia and Paoli, Gloria},
  year = {2024},
  pages = {31}
}

@article {ABCMP,
    author = {{Amato}, Vincenzo and {Barbato}, Rosa and {Cito}, Simone and {Masiello}, Alba Lia and {Paoli}, Gloria},
    title = "{A quantitative Talenti-type comparison result with Robin boundary conditions}",
   JOURNAL = {

https://doi.org/10.48550/arXiv.2511.11316
},
      YEAR = {2025},
       DOI = {10.48550/arXiv.2511.11316},
       URL = {https://arxiv.org/abs/2511.11316},
}

@article{AB,
  title     = "Quantitative comparison results for first-order
               {Hamilton-Jacobi} equations",
  author    = "Amato, Vincenzo and Barbato, Luca",
  journal   = "Acta Appl. Math.",
  publisher = "Springer Science and Business Media LLC",
  volume    =  200,
  number    =  1,
  year      =  2025,
  copyright = "https://creativecommons.org/licenses/by/4.0",
  language  = "en"
}

@article {MS,
    AUTHOR = {Masiello, Alba Lia and Salerno, Francesco},
     TITLE = {A quantitative result for the {$k$}-{H}essian equation},
   JOURNAL = {Nonlinear Anal.},
  FJOURNAL = {Nonlinear Analysis. Theory, Methods \& Applications. An
              International Multidisciplinary Journal},
    VOLUME = {255},
      YEAR = {2025},
     PAGES = {Paper No. 113776},
      ISSN = {0362-546X,1873-5215},
   MRCLASS = {52A39 (35B35 35J60 35J96)},
  MRNUMBER = {4867241},
       DOI = {10.1016/j.na.2025.113776},
       URL = {https://doi.org/10.1016/j.na.2025.113776},
}

@article{AL,
  title = {Sharp Quantitative {T}alenti Inequality in Particular Cases},
  volume = {250},
  ISSN = {1432-0673},
  url = {http://dx.doi.org/10.1007/s00205-026-02179-3},
  DOI = {10.1007/s00205-026-02179-3},
  number = {3},
  journal = {Archive for Rational Mechanics and Analysis},
  publisher = {Springer Science and Business Media LLC},
  author = {Acampora,  Paolo and Lamboley,  Jimmy},
  year = {2026},
  month = apr 
}

@article {F,
    AUTHOR = {Fusco, Nicola},
     TITLE = {The quantitative isoperimetric inequality and related topics},
   JOURNAL = {Bull. Math. Sci.},
  FJOURNAL = {Bulletin of Mathematical Sciences},
    VOLUME = {5},
      YEAR = {2015},
    NUMBER = {3},
     PAGES = {517--607},
      ISSN = {1664-3607,1664-3615},
   MRCLASS = {49Q20 (28A75 52A40 53A10)},
  MRNUMBER = {3404715},
MRREVIEWER = {Georgios\ Psaradakis},
       DOI = {10.1007/s13373-015-0074-x},
       URL = {https://doi.org/10.1007/s13373-015-0074-x},
}

@book{H,
url = {https://doi.org/10.1515/9783110550887},
title = {Shape optimization and spectral theory},
author = {Henrot, Antoine},
publisher = {De Gruyter Open Poland},
address = {Warsaw, Poland},
doi = {doi:10.1515/9783110550887},
isbn = {9783110550887},
year = {2017},
lastchecked = {2026-05-20}
}

@incollection {BW,
    AUTHOR = {Bhattacharya, Tilak and Weitsman, Allen},
     TITLE = {Estimates for {G}reen's function in terms of asymmetry},
 BOOKTITLE = {Applied analysis ({B}aton {R}ouge, {LA}, 1996)},
    SERIES = {Contemp. Math.},
    VOLUME = {221},
     PAGES = {31--58},
 PUBLISHER = {Amer. Math. Soc., Providence, RI},
      YEAR = {1999},
      ISBN = {0-8218-0673-4},
   MRCLASS = {31A05 (30C20)},
  MRNUMBER = {1647193},
MRREVIEWER = {Peter\ W.\ Day},
       DOI = {10.1090/conm/221/03117},
       URL = {https://doi.org/10.1090/conm/221/03117},
}

@article {FMP,
    AUTHOR = {Fusco, Nicola and Maggi, Francesco and Pratelli, Aldo},
     TITLE = {Stability estimates for certain {F}aber-{K}rahn, isocapacitary
              and {C}heeger inequalities},
   JOURNAL = {Ann. Sc. Norm. Super. Pisa Cl. Sci. (5)},
  FJOURNAL = {Annali della Scuola Normale Superiore di Pisa. Classe di
              Scienze. Serie V},
    VOLUME = {8},
      YEAR = {2009},
    NUMBER = {1},
     PAGES = {51--71},
      ISSN = {0391-173X,2036-2145},
   MRCLASS = {49R05 (26D20 35J60 35P30 49J40)},
  MRNUMBER = {2512200},
MRREVIEWER = {Marino\ Belloni},
}

@inproceedings {N,
    AUTHOR = {Nadirashvili, Nikolai},
     TITLE = {Conformal maps and isoperimetric inequalities for eigenvalues
              of the {N}eumann problem},
 BOOKTITLE = {Proceedings of the {A}shkelon {W}orkshop on {C}omplex
              {F}unction {T}heory (1996)},
    SERIES = {Israel Math. Conf. Proc.},
    VOLUME = {11},
     PAGES = {197--201},
 PUBLISHER = {Bar-Ilan Univ., Ramat Gan},
      YEAR = {1997},
   MRCLASS = {35J25 (35P15)},
  MRNUMBER = {1476715},
MRREVIEWER = {G\"unter\ Berger},
}

@article {CCLMP,
    AUTHOR = {Carbotti, Alessandro and Cito, Simone and La Manna, Domenico
              Angelo and Pallara, Diego},
     TITLE = {Stability of the {G}aussian {F}aber-{K}rahn inequality},
   JOURNAL = {Ann. Mat. Pura Appl. (4)},
  FJOURNAL = {Annali di Matematica Pura ed Applicata. Series IV},
    VOLUME = {203},
      YEAR = {2024},
    NUMBER = {5},
     PAGES = {2185--2198},
      ISSN = {0373-3114,1618-1891},
   MRCLASS = {35P15 (49Q10 49R05)},
  MRNUMBER = {4797290},
MRREVIEWER = {Yunfeng\ Shi},
       DOI = {10.1007/s10231-024-01441-3},
       URL = {https://doi.org/10.1007/s10231-024-01441-3},
}
\Addresses
\end{document}